\documentclass[a4paper,11pt,draft]{amsart}
\usepackage{amssymb,amscd}
\usepackage[cp1251]{inputenc}
\usepackage[T2A]{fontenc}
\theoremstyle{plain}

\newtheorem{theo}{Theorem}[section]
\newtheorem{prop}[theo]{Proposition}
\newtheorem{lemm}[theo]{Lemma}
\newtheorem{coro}[theo]{Corollary}

\theoremstyle{definition}

\newtheorem{exam}[theo]{Example}

\theoremstyle{remark}
\newtheorem*{rema}{Remark}

\numberwithin{equation}{section}

\def\k{\mathbf{k}}

\DeclareMathOperator{\Tor}{Tor}

\def\le{\leqslant}
\def\ge{\geqslant}

\newcommand{\zp}{\mathcal Z_P}

\DeclareMathOperator{\vc}{\mbox{\textit{vc}}}

\begin{document}

\title[Stanley--Reisner rings and moment-angle manifolds]{Stanley--Reisner rings of generalised truncation polytopes and their moment-angle manifolds}
%\date{\today}

\author{Ivan Limonchenko}

\thanks{The work was supported by
grant MD-111.2013.1 from the President of Russia and the Russian Foundation for Basic Research grants nn. 14-01-31398 and 14-01-00537-a.\\
%\copyright\,  I.Yu. Limonchenko
}

%\subjclass[2000]{14L30, 14M25,57S25}
\address{Department of Geometry and Topology,
Faculty of Mathematics and Mechanics, Moscow State University,
Leninskiye Gory, Moscow 119992, Russia}
\email{iylim@mail.ru}

%\keywords{Betti numbers, moment-angle-complexes, Stasheff
%polytopes, truncation polytopes}

\begin{abstract}
We consider simple polytopes $P=\vc^{k}(\Delta^{n_{1}}\times\ldots\times\Delta^{n_{r}})$, for $n_1\ge\ldots\ge n_r\ge 1,r\ge 1,k\ge 0$, that is, $k$-vertex cuts of a product of simplices, and call them {\textit{generalized truncation polytopes}}. For these polytopes we describe the cohomology ring of the corresponding moment-angle manifold $\zp$ and explore some topological consequences of this calculation. We also examine minimal non-Golodness for their Stanley--Reisner rings and relate it to the property of $\zp$ being a connected sum of sphere products.     
\end{abstract}

\dedicatory{Dedicated to Professor Victor M. Buchstaber on the occasion of his 70th birthday}

\maketitle

\section{Introduction}
We denote by $\k$ the base field of zero characteristic or the ring of integers. Let $\k[v_{1},\ldots,v_{m}]$ be the graded polynomial algebra on $m$ variables, $\deg(v_{i})=2$, and let $\Lambda[u_{1},\ldots,u_{m}]$ be the exterior algebra,
$\deg(u_{i})=1$. The \emph{face ring} (also known as the
\emph{Stanley--Reisner ring}) of a simplicial complex $K$
on a vertex set $[m]=\{1,\ldots,m\}$ is the quotient ring
$$
   \k[K]=\k[v_{1},\ldots,v_{m}]/\mathcal I_K
$$
where $\mathcal I_K$ is the ideal generated by those square free
monomials $v_{i_{1}}\cdots{v_{i_{k}}}$ for which
$\{i_{1},\ldots,i_{k}\}$ is not a simplex in $K$. We refer to
$\mathcal I_{K}$ as the \emph{Stanley--Reisner ideal} of~$K$.

Note that $\k[K]$ is a module over $\k[v_{1},\ldots,{v_{m}}]$ via
the quotient projection. The dimensions of the bigraded components
of the $\Tor$-groups,
$$
  \beta^{-i,2j}(\k[K]):=\dim_{\k}\Tor^{-i,
  2j}_{\k[v_{1},\ldots,v_{m}]}\bigl(\k[K],\k\bigr),\quad
   0\le{i,j}\le{m},
$$
are known as the \emph{bigraded Betti numbers} of~$\k[K]$,
see~\cite{S} and \cite[\S3.3]{B-P}. They are important invariants
of the combinatorial structure of $K$. 

A face ring $\k[K]$ is called {\textit{Golod}}
if the multiplication and all higher Massey operations
in $\Tor_{\k[v_{1},\ldots,{v_{m}}]}\bigl(\k[K],\k\bigr)$
are trivial. If so, the complex $K$ is also called {\textit{Golod}}.
If $K$ itself is not Golod but deleting any vertex $v$ from $K$
turns the restricted complex $K-\,v$ into a Golod one, then $\k[K]$ and $K$ itself are called {\textit{minimally non-Golod}}.  

The {\textit{moment-angle complex}} of a simplicial complex $K$ is a CW-complex $\mathcal Z_K=\bigcup\limits_{I\in K}\Bigl(\prod\limits_{i\in I}D^2\times\prod\limits_{i\notin I}S^1\Bigr)$ (viewed as a subcomplex in a unit disk $(D^{2})^{m}$). The $\Tor$-groups and the bigraded Betti numbers of $\k[K]$ acquire a topological interpretation by means of the following result on the cohomology of $\mathcal Z_K$:

\begin{theo}[{\cite[Theorem 8.6]{B-P} or \cite[Theorem 4.7]{P}}]\label{zkcoh}
The cohomology algebra of the moment-angle manifold $\mathcal Z_K$ is given
by the isomorphisms
\[
\begin{aligned}
  H^*(\mathcal Z_K;\k)&\cong\Tor_{\k[v_1,\ldots,v_m]}(\k[K],\k)\\
  &\cong H\bigl[\Lambda[u_1,\ldots,u_m]\otimes \k[K],d\bigr]\\
  &\cong \bigoplus\limits_{I\subset [m]}\widetilde{H}^*(K_{I}),
\end{aligned}
\]
where bigrading and differential in the cohomology of the differential
bigraded algebra are defined by
\[
  \mathop{\mathrm{bideg}} u_i=(-1,2),\;\mathop{\mathrm{bideg}} v_i=(0,2);\quad
  du_i=v_i,\;dv_i=0.
\]
In the third row, $\widetilde{H}^*(K_{I})$ denotes the reduced simplicial cohomology of the {\textit{full subcomplex}} $K_{I}$ of $K$ (the restriction of $K$ to $I\subset [m]$). The last isomorphism is the sum of isomorphisms 
$$H^p(\mathcal Z_K)\cong\sum\limits_{I\subset [m]}\widetilde{H}^{p-|I|-1}(K_{I}),$$
and the ring structure (the Hochster ring) is given by the maps
%\begin{equation}\label{star}
$$
H^{p-|I|-1}(K_{I})\otimes H^{q-|J|-1}(K_{J})\to H^{p+q-|I|-|J|-1}(K_{I\cup J}),\eqno (*)
$$
%\end{equation}
which are induced by the canonical simplicial maps $K_{I\cup J}\hookrightarrow K_{I}*K_{J}$ (join of simplicial complexes) for $I\cap J=\varnothing$ and {\bf{zero}} otherwise.
\end{theo}

Let $P$ be a simple $n$-dimensional convex polytope with $m$ {\textit{facets}} (faces of codimension 1) $F_{1},\ldots,F_{m}$. Denote by $K_P$ the boundary $\partial P^*$ of the dual simplicial polytope. It can be viewed as a $(n-1)$-dimensional simplicial complex on the set $[m]$, whose simplices are subsets
$\{i_1,\ldots,i_k\}$ such that $F_{i_1}\cap\ldots\cap
F_{i_k}\ne\varnothing$ in~$P$.

We call $\zp=\mathcal Z_{K_P}$ the {\textit{moment-angle manifold}} of $P$. By \cite[Lemma 7.2]{B-P}, $\zp$ is a smooth closed manifold of
dimension $m+n$. 

Therefore, by Theorem~\ref{zkcoh} cohomology of $\zp$ acquires a bigrading and the topological Betti numbers $b^{q}(\zp)=\dim_{\k}H^{q}(\zp;\k)$
satisfy
%\begin{equation}\label{bp}
$$
  b^{q}(\zp)=\sum\limits_{-i+2j=q}\beta^{-i,2j}(\k[K_P]).\eqno (1.3)
$$
We denote
\[
  \beta^{-i,2j}(P):=\beta^{-i,2j}(\k[K_{P}]).
\]

Poincar\'e duality in cohomology of $\zp$ respects the bigrading:
\begin{theo}[{\cite[Theorem 8.18]{B-P}}]\label{Poinc}
The following formula holds:
$$
   \beta^{-i,2j}(P)=\beta^{-(m-n)+i,2(m-j)}(P).
$$
\end{theo}

From now on we shall drop the coefficient field $\k$ from the
notation of (co)homology groups. The following classical
result can be also obtained as a corollary of Theorem~\ref{zkcoh}:
\begin{theo}[Hochster, see~{\cite[Cor. 8.8]{B-P}}]\label{hoh}
Let $K=K_P$. We have:
$$
   \beta^{-i,2j}(P)=\sum\limits_{J\subset{[m]},|J|=j}\dim
   \widetilde{H}^{j-i-1}(K_{J}).
$$
\end{theo}

The structure of this paper is as follows. Complete calculation of bigraded Betti numbers for generalized truncation polytopes is given in
Section~2. We then apply our calculation to describing the topology of the corresponding moment-angle manifold. For truncation polytopes $P=vc^{k}(\Delta^{n})$ with $n\ge 2,k\ge 0$ this was done in~\cite{Li}. Our results on the bigraded Betti numbers agree with partial description of the diffeomorphism types of $\zp$ for generalized truncation polytopes with $r=1,2$ from~\cite{G-LdM}.  In Section~3 we prove a criterion of when a face ring of an arbitrary generalized truncation polytope is minimally non-Golod, the latter property was proved for truncation polytopes in~\cite[Theorem 6.19]{B-J}. We also relate our results to a conjecture of~\cite{G-P-T-W}, cohomological rigidity for moment-angle manifolds and topological invariance problems for bigraded Betti numbers.

The author is deeply grateful to his scientific adviser Taras Panov for
many helpful discussions and advice which was always so kindly
proposed during this work and whenever. 
Many thanks to Hiraku Abe, Anton Ayzenberg and Suyoung Choi for their comments and suggestions and to Mikiya Masuda for organizing the Toric Topology meetings in Osaka where the author gave his first related talk. 

\section{Cohomology rings of moment-angle manifolds for generalized truncation polytopes}

Let $P$ be a simple $d$-dimensional polytope and $v\in{P}$ its vertex. Choose
a hyperplane $H$ such that $H$ separates $v$ from the other
vertices and $v$ belongs to the positive halfspace $H_\ge$
determined by~$H$. Then $P\cap H_\ge$ is a $d$-simplex, and
$P\cap H_\le$ is a simple polytope, which we refer to as a
\emph{vertex truncation (cut)} of~$P$. When the choice of the cut vertex is
clear or irrelevant we use the notation $\vc(P)$. We also use the
notation $\vc^k(P)$ for a polytope obtained from $P$ by iterating
the vertex cut operation $k$ times. The combinatorial type of $P$ may depend on the choice of truncated vertices but it is not important for us here.

As an example of this procedure, we consider the polytope
$P=\vc^{k}(\Delta^{n_{1}}\times\ldots\times\Delta^{n_{r}})$, for $n_1\ge\ldots\ge n_r\ge 1,r\ge 1,k\ge 0$, where $\Delta^n$ is an $n$-simplex. This becomes a truncation polytope (dual to a stacked polytope, see~\cite{Li}) if and only if $r=1$ or $r=2,\,n_{2}=1$, here we refer to such $P$ as a \emph{generalized truncation polytope of type} $(k;n_1,\ldots,n_r)$ or  a $(k;n_1,\ldots,n_r)$-polytope; it has $m=n_{1}+\ldots+n_{r}+r+k$ facets and dimension $d=n_{1}+\ldots+n_{r}$. 

The bigraded Betti numbers for stacked polytopes were calculated
completely in~\cite{T-H}. We begin with a theorem that gives a complete calculation of bigraded Betti numbers for the case $r=2$: 

\begin{theo}\label{maintr3}
For a generalized truncation polytope $P=vc^{k}(\Delta^{n_{1}}\times \Delta^{n_{2}})$ with $n_{1}\ge{n_{2}}>1,k\ge 0$, the bigraded Betti numbers of $\zp$ are given by the following formulae $(1\le i\le k+1)$.

\begin{itemize}
\item[(a)]{$\beta^{-i,2(i+n)}(P)=2\binom{k}{i-1}$, if $n_{1}=n_{2}=n$}
\item[(b)]{$\beta^{-i,2(i+n_{1})}(P)=\beta^{-i,2(i+n_{2})}(P)=\binom{k}{i-1}$, if $n_1 > n_2$} 
\item[(c)]{$\beta^{-i,2(i+1)}(P)=\beta^{-(k+2-i),2(k+1-i+n_{1}+n_{2})}(P)=i\binom{k+2}{i+1}-\binom{k}{i-1}$} 
\item[(d)]{$\beta^{0,0}(P)=\beta^{-(m-d),2m}(P)=1$}
\end{itemize}
The other bigraded Betti numbers are \textbf{zero}, where $m$ is the number of facets of $P$ and $d$ is its dimension.
\end{theo}

%\begin{rema}
%The first of the above formulae was proved in~\cite{Ch-K} combinatorially.
%\end{rema}
\begin{proof}
To proceed by induction on the number of vertex truncations $k$ we should analyze the behavior of bigraded Betti numbers under
a single vertex cut. Let $P$ be an arbitrary simple polytope and
$P'=\vc(P)$. We denote by $Q$ and $Q'$ the dual simplicial
polytopes respectively, and denote by $K$ and $K'$ their boundary
simplicial complexes. Then $Q'$ is obtained by adding a pyramid
with a vertex $v$ over a facet $F$ of~$Q$. We also denote by $V$,
$V'$ and $V(F)$ the vertex sets of $Q$, $Q'$ and $F$ respectively,
so that $V'=V\cup{v}$.

The proof of (c) is based on the following lemma:
\begin{lemm}[{\cite{T-H},\cite{Li}}]\label{betapp'}
Let $P$ be a simple $d$-polytope with $m$ facets and $P'=\vc(P)$.
Then
$$
 \beta^{-i,2(i+1)}(P')=\binom{m-d}{i}+\beta^{-(i-1),2i}(P)+\beta^{-i,2(i+1)}(P).
$$
\end{lemm}
\begin{proof}
This is done by a direct application of Theorem~\ref{hoh} for $j=i+1$.
\end{proof}

Now formula (c) of Theorem~\ref{maintr3} follows by induction
on the number of cut vertices, using the fact that
$\beta^{-i,2(i+1)}(\Delta^{n_{1}}\times\Delta^{n_{2}})=0$, $n_{1}\ge n_{2}>1$ for all~$i$, Lemma~\ref{betapp'} and the bigraded Poincare duality, see Theorem~\ref{Poinc}. Formula (d) is valid for all simple polytopes.\\

In (b) the first equality also follows from the bigraded Poincare duality, see Theorem~\ref{Poinc}. It remains to prove (a) and the second equality of (b). We use the following lemma:\\

%The proof of the first formula and the rest of the second formula relies on the %following lemma.

\begin{lemm}\label{zero}
Let $P$ be a generalized truncation polytope of type $(k;n_1,n_2)$, $K$ the boundary complex of the dual simplicial polytope, $V$ the vertex set of~$K$, and $W$ a nonempty proper subset of~$V$. Then
\begin{itemize}
\item[(i)]{$\widetilde{H}_{i}(K_{W})=0$ for $i\neq{0,n_{1}-1,n_{2}-1,d-2};$}
\item[(ii)]{
For $i=n_{1}-1$ or $n_{2}-1$ the homology group $\widetilde{H}_{i}(K_{W})$ is nontrivial (and then $\cong{\k}$) if and only if $W=V(\Delta^{n_1})\cup{NV_1}$ or $V(\Delta^{n_2})\cup{NV_2}$, where $NV_{1,2}\subset{NV}$, $NV$ is a set of 'new vertices' of $P$, $|NV|=k$.
}
\end{itemize}
\end{lemm}
\begin{proof}

The proof is by induction on the number $k$ of vertex truncations. We begin with the proof of (i).\\

If $k=0$ then $P=\Delta^{n_1}\times\Delta^{n_2}$, and $K_P=(\partial\Delta^{n_{1}})*(\partial\Delta^{n_2})\cong S^{n_{1}-1}*S^{n_{2}-1}\cong S^{n_{1}+n_{2}-1}$. Therefore, by the definition of join, $K_{W}$ is either contractible or homotopy equivalent to a sphere of dimension either $n_{1}-1$ or $n_{2}-1$ for every proper subset $W\subset{V}$.\\

To make the induction step we consider $V'=V\cup{v}$, $V_{1}=V(\Delta^{n_1})$, $V_{2}=V(\Delta^{n_2})$ and $V(F)$ as
in the beginning of the proof of Theorem~\ref{maintr3}. Assume the
statement is proved for $V$ and let $W$ be a proper subset
of~$V'$.\\

We consider 5 cases, following~\cite{Li}.\\

%\smallskip

\noindent\emph{Case 1:} $v\in{W},\; W\cap{V(F)}\neq{\varnothing}.$\\

If $V(F)\subset{W}$, then $K'_{W}$ is a subdivision of
$K_{W-\{v\}}$. It follows that
$\widetilde{H}_{i}(K'_{W})\cong{\widetilde{H}_{i}(K_{W-\{v\}})}$.

If $W\cap{V(F)}\neq{V(F)}$, then we have
$$
   K'_{W}=K_{W-\{v\}}\cup{K'_{W\cap{V(F)\cup{\{v\}}}}},\quad
   K_{W-\{v\}}\cap{K'_{W\cap{V(F)\cup{\{v\}}}}}=K_{W\cap{V(F)}},
$$
and both $K_{W\cap{V(F)}}$ and $K'_{W\cap{V(F)\cup{\{v\}}}}$ are
contractible. From the Mayer--Vietoris exact sequence we again
obtain
$\widetilde{H}_{i}(K'_{W})\cong{\widetilde{H}_{i}(K_{W-\{v\}})}$.\\

%\smallskip

\noindent\emph{Case 2:} $v\in{W},\; W\cap{V(F)}=\varnothing.$\\

In this case it is easy to see that
$K'_{W}=K_{W-\{v\}}\sqcup{\{v\}}.$ It follows that
$$
   \widetilde{H}_{i}(K'_{W})\cong\begin{cases}
   \widetilde{H}_{i}(K_{W-\{v\}})\oplus{\k},&\text{for $i=0;$}\\
   \widetilde{H}_{i}(K_{W-\{v\}}),&\text{for $i>0.$}
                                     \end{cases}
$$

%\smallskip

\noindent\emph{Case 3:}  $W=V'-\{v\}=V.$\\

Then $K'_{W}$ is a triangulated $(d-1)$-disk and therefore
contractible.\\

%\smallskip

\noindent\emph{Case 4:} $v\not\in{W},\; V(F)\subset{W},\;W\neq{V}.$\\

We have
$$
   K_{W}=K'_{W}\cup{F},\quad
   K'_{W}\cap{F}=\partial{F},
$$
where $\partial{F}$ is the boundary of the facet $F$. Since
$\partial{F}$ is a triangulated ${(d-2)}$-sphere and~$F$ is a
triangulated $(d-1)$-disk, the Mayer--Vietoris homology sequence
implies that
$$
   \widetilde{H}_{i}(K'_{W})\cong\begin{cases}
   \widetilde{H}_{i}(K_{W}),&\text{for $i<d-2;$}\\
   \widetilde{H}_{i}(K_{W})\oplus{\k},&\text{for $i=d-2.$}
                                     \end{cases}
$$\\

%\smallskip

\noindent\emph{Case 5:} $v\not\in{W},\; V(F)\not\subset{W}.$\\

In this case we have $K'_{W}\cong{K_{W}}.$

In all cases we obtain
$$
   \widetilde{H}_{i}(K'_{W})\cong\widetilde{H}_{i}(K_{W-\{v\}})=0\quad
   \text{for }i\neq{0,n_{1}-1,n_{2}-1,d-2},
$$
which finishes the proof of (i) by induction.\\

To prove (ii) note that from the above proof and the definition of join it follows that $\widetilde{H}_{n_{1}-1}(K_{W})$ is nontrivial if and only if $V(W)=V_{1}\cup{NV_1}$ for some $NV_1\subset{NV},|NV_1|\ge 0$, that is $V_{1}\subset{W}$ and $V_{2}\cap{W}=\varnothing$. In the latter case $\widetilde{H}_{n_{1}-1}(K_{W})\cong{\k}$. The case of $(n_{2}-1)$-dimensional homology groups is similar.
\end{proof}

Now statements (a) and (b) of Theorem~\ref{maintr3} follow from Theorem~\ref{hoh} (if $K_J$ has $j=i+n_{1}$ vertices then $|V_1|=n_{1}+1$ implies that any $j-|V_1|=i-1$ {\bf{new}} vertices uniquely determine such $K_J$).\\

Vanishing of the other bigraded Betti numbers in Theorem~\ref{maintr3} follows from Theorem~\ref{hoh} and Lemma~\ref{zero}.
\end{proof}

Our main result in this section gives a complete calculation of the bigraded Betti numbers of all generalized truncation polytopes:

\begin{theo}\label{maintr}
Let $P$ be a $(k;n_1,\ldots,n_r)$-polytope with $r\ge 1$. Denote by $a$ the number of '1's in the set $\{n_1,\ldots,n_r\}$. 
Then the bigraded Betti numbers of $P$ are given by the following formulae $(1\le i\le k+r-1, 1<l<d-1)$.
\begin{itemize}
\item[(a)]{$\beta^{-i,2(i+l)}(P)=
\sum\limits_{\{n_{i_1},\ldots,n_{i_s}\}\subset\{n_1,\ldots,n_r\}:\, l=n_{i_1}+\ldots+n_{i_s}}\binom{k}{i-s}.$}
%We have:
\item[(b)]{
%\begin{multline*}
$\beta^{-i,2(i+1)}(P)=\beta^{-(k+r-i),2(d+k+r-i-1)}(P)=$\\
\hspace{2cm}
$=k\binom{k+r-1}{i}-\binom{k}{i+1}+a\binom{k}{i-1}.$%
\\
%\end{multline*}\\
}
\item[(c)]{$\beta^{0,0}(P)=\beta^{-(m-d),2m}(P)=1.$}
\end{itemize}
The other bigraded Betti numbers are \textbf{zero}
(we assume $\binom{b}{c}=0$ if $b<c$ or one of them is negative).
\end{theo}
\begin{proof}
We use the notation of Theorem~\ref{maintr3}.

To prove (a) we modify the statement of Lemma~\ref{zero} for the general case as follows. From the proof of Lemma~\ref{zero} and Theorem~\ref{hoh} it is clear that $\beta^{-i,2j}(P)=0$ for all $j:\,j-i\neq{n_{i_1}+\ldots+n_{i_s}},\{n_{i_1},\ldots,n_{i_s}\}\subset\{n_1,\ldots,n_r\}$. For other $1<l=j-i<d-1$, the full subcomplexes $K_W$ with $W=V(\Delta^{n_p})\cup{NV_p}$, $NV_p\subset NV$, for some $1\le p\le r$, and {\bf{all their joins}} give all nontrivial homology groups $\widetilde{H}_{l-1}(K_{W})$.     
Here we used the fact that $S^{p}*S^{q}\cong S^{p+q+1}$. This implies that $j=i+l=(n_{1}+1)+\ldots+(n_{s}+1)+(i-s)$ and so any $(i-s)$ {\bf{new}} vertices give +1 in the sum (a) when $i$ and $l$ are fixed.\\

For the statement (b), the third term of the sum appears from the argument similar to the one for formula (a) (we have $a$ summands of the type $\binom{k}{i-s}$ with $s=1$); the first two terms appear from Lemma~\ref{betapp'} by induction on the number of vertex truncations $k$.\\

The rest of Theorem~\ref{maintr} follows from Theorem~\ref{hoh}.

\end{proof}

\begin{rema}
Note that formula from Theorem~\ref{maintr} (b) in the case $a=0$, $r=2$ gives $\beta^{-i,2(i+1)}(P)=k\binom{k+r-1}{i}-\binom{k}{i+1}=
i\binom{k+2}{i+1}-\binom{k}{i-1}$, in accordance with Theorem~\ref{maintr3} (c).
\end{rema}

We also include the results about the
bigraded Betti numbers of truncation polytopes ($n\ge 3$) and polygons ($n=2$) for further use:

\begin{prop}[{\cite{T-H},\cite{Li}}]\label{trunc}

Let $P=\vc^{k}(\Delta^{n})$ be a truncation polytope. Then for
$n\ge{3}$ the bigraded Betti numbers are given by the following
formulae:
\begin{itemize}
\item[(a)]{$\beta^{-i,2(i+1)}(P)=i\binom{k+1}{i+1},$}
\item[(b)]{$\beta^{-i,2(i+n-1)}(P)=(k+1-i)\binom{k+1}{k+2-i},$}
\item[(c)]{$\beta^{-i,2j}(P)=0,\quad\text{for } i+1< j< i+n-1.$}
\item[(d)]{$\beta^{0,0}(P)=\beta^{-(m-n),2m}(P)=1.$}
\end{itemize}
The other bigraded Betti numbers are zero.
\end{prop}

\begin{prop}[{see~\cite[Example~8.21]{B-P}}]\label{polyg}

If $P=\vc^{k}(\Delta^{2})$ is an $(k+3)$-gon, then
\begin{itemize}
\item[(a)]{$\beta^{-i,2(i+1)}(P)=i\binom{k+1}{i+1}+(k+1-i)\binom{k+1}{k+2-i},$}
\item[(b)]{$\beta^{0,0}(P)=\beta^{-(k+1),2(k+3)}(P)=1.$}
\end{itemize}
The other bigraded Betti numbers are zero.
\end{prop}

\begin{coro}
The bigraded Betti numbers of a generalized truncation polytope
$P$ depend only on the dimension and the number
of facets of $P$ and do not depend on its combinatorial type.
Moreover the numbers $\beta^{-i, 2(i+1)}(P)$ do not depend on the
dimension~$d$.
\end{coro}

\begin{coro}
In the class of all generalized truncation polytopes $P$ of all possible types $(k;n_1,\ldots,n_r)$ the set of all bigraded Betti numbers $\{\beta^{-i,2j}(P)\}$ uniquely determines the type. For generalized truncation polytopes $P$ and $Q$ two bigraded rings $H^{*,*}(\mathcal Z_P)$ and $H^{*,*}(\mathcal Z_Q)$ are isomorphic if and only if all their bigraded Betti numbers are equal. 
\end{coro}

\begin{exam}

Consider a simple polytope $P=vc^1(\Delta^4\times\Delta^3\times\Delta^2)$, for which $d=9$, $m=13$. We calculate the bigraded Betti numbers of $P$ using software package \emph{Macaulay 2}, see~\cite{Mac}. 

The table below has $d-1$ rows and $m-d-1$ columns. The number
in the intersection of the $l$th row and the $i$th column is
$\beta^{-i,2(i+l)}(P)$, where $ 1\le{i}\le{m-d-1}$ and
$2\le{j=i+l}\le{m-2}$. The other bigraded Betti numbers are zero
except for $\beta^{0,0}(P)=\beta^{-(m-d),2m}(P)=1$,
see~\cite[Ch.8]{B-P}:
\begin{center}
\noindent\\[2pt]
\small
\begin{tabular}{|c|c|c|c|}
\hline $i,l$ & $i=1$ & $i=2$ & $i=3$\tabularnewline
\hline $l=1$ & 3 & 3 & 1\tabularnewline 
\hline $l=2$ & 1 & 1 & 0\tabularnewline 
\hline $l=3$ & 1 & 1 & 0\tabularnewline
\hline $l=4$ & 1 & 1 & 0\tabularnewline
\hline $l=5$ & 0 & 1 & 1\tabularnewline 
\hline $l=6$ & 0 & 1 & 1\tabularnewline
\hline $l=7$ & 0 & 1 & 1\tabularnewline
\hline $l=8$ & 1 & 3 & 3\tabularnewline
\hline
\end{tabular}
\end{center}
\medskip
\end{exam}

Theorem~\ref{zkcoh}, the result and proof of Theorem~\ref{maintr} together with the obvious fact that $K_P$ is free of torsion give us the complete description of the cohomology ring of $\zp$ for all $(k;n_1,\ldots,n_r)$-polytopes $P$. We will use it in the next section to examine minimal non-Golodness for $K_P$ and some topological properties of $\zp$ for generalized truncation polytopes.

\section{Minimal non-Golodness of $K_P$ and the topology of $\zp$}

\begin{prop}\label{glue}
Suppose $K=K_{1}\cup_{\sigma}K_{2}$ is a simplicial complex obtained from two Golod complexes $K_1$ and $K_2$ by gluing along a common simplex. Then $K$ is also Golod.
\end{prop}
\begin{proof}
Following the description of multiplication in the Hochster ring from Theorem~\ref{zkcoh} (*) and using~\cite[Proposition 3.2.10, formula (3.11)]{TT}, we denote by $\alpha_L$ the basis cochain in $C^{p}(K_I)$ corresponding to an oriented $p$-simplex $L\subset I$, $I\subset V(K)$. Then the product of cochains $\alpha_L$ and $\alpha_M$, $M\subset J$, $J\subset V(K)$, $I\cap J=\varnothing$ is nontrivial, if and only if $L\cup M$ is a simplex in $K_{I\cup J}$. Let us also denote $I_k=I\cap V(K_k)$ and $J_k=J\cap V(K_k)$ for $k=1,2$.

Suppose for a pair of disjoint subsets $I,J\subset V(K)$ we have a nontrivial product by the formula Theorem~\ref{zkcoh} (*). By the above argument for some simplices $L\subset I$ and $M\subset J$ we have that $L\cup M$ is also a simplex in $K_{I\cup J}$. Then by the gluing construction we get that either $L\subset I_1$, $M\subset J_1$, $L\cup M\subset K_1$ or $L\subset I_2$, $M\subset J_2$, $L\cup M\subset K_2$.

But the two complexes $K_1$ and $K_2$ are Golod, so the product of $\alpha_L$ and $\alpha_M$ is a coboundary in $C^{*}(K)$ and so $K$ is Golod itself. 
\end{proof}

Now we prove our main result in this section:

\begin{theo}\label{maingolod}
For $P=\vc^{k}(\Delta^{n_{1}}\times\ldots\times\Delta^{n_{r}})$ with $n_1\ge\ldots\ge n_r\ge 1,r\ge 1,k\ge 0$, its face ring $K_P$ is minimally non-Golod if and only if $r=1,2$ and $P$ is not a simplex.
\end{theo}
\begin{proof}

The proof is by induction on the number of new vertices $k$.\\

Suppose $r=2$. We use notation from the proof of Theorem~\ref{maintr}.\\

If $k=0$ then $P=\Delta^{n_1}\times\Delta^{n_2}$ and $K_P=(\partial\Delta^{n_{1}})*(\partial\Delta^{n_2})$. Suppose we delete a vertex $v$ that belongs to $\Delta^{n_2}$ from the complex $K=K_P$. Then we have: $K'=K-\,v=(\partial\Delta^{n_1}*\Delta^{n_{2}-1})$ and $\mathcal Z_{K'}=\mathcal Z_{\partial\Delta^{n_1}}\times\mathcal Z_{\Delta^{n_{2}-1}}=S^{2n_{1}+1}\times D^{2n_{2}}$ which is homotopy equivalent to a sphere and so $K'$ is Golod by Theorem~\ref{zkcoh}. \\

For $k\ge 1$, we denote $Q=\vc(P)$, $K=K_Q$, $KK=K_P$, $V_{1}=V(\Delta^{n_1})$, $V_{2}=V(\Delta^{n_2})$ and use the description of multiplication from Theorem~\ref{zkcoh} (see formula (*)) and Lemma~\ref{zero} for generalized truncation polytopes.\\

{\noindent{\bf A.}} Suppose we delete a vertex $v\in NV$ which is a vertex of exactly one pyramid (over a facet $F$ of $P^*$).\\

By the description of multiplication in Theorem~\ref{zkcoh} (*) a nontrivial product in cohomology ring of $\mathcal Z_{K'}$ arises from a pair of subsets $I,J\subset{V(K')}$, $I\cap{J}=\varnothing$ and $I\cup J=V(K_P)$ and we may assume (Lemma~\ref{zero}) that $V_{1}\subset{I}$, $V_{2}\subset{J}$. We consider the following 2 cases: \\

{\noindent{\textit{Case 1:}}} $I,J\cap V(F)\neq\varnothing$.\\

We know by induction that $KK=K_P$ is minimally non-Golod and any nontrivial product in $H^*(\mathcal Z_{K'})$ should arise from $H^*(\mathcal Z_{K_{P}})$ (here we use Theorem~\ref{zkcoh}, formula (*)). However the full subcomplexes $K'_{I}=KK_I$ and $K'_{J}=KK_J$ cannot give rise to a nontrivial product, otherwise in the formula (*) we get a $(n_{1}-1)+(n_{2}-1)+1=(d-1)$-dimensional cohomology class of a full subcomplex of $K'$, which is obviously trivial.\\

{\noindent{\textit{Case 2:}}} $V(F)$ is contained in one of the $I,J$, say $V(F)\subset{I}$.\\

Then $H^{d-2}(K_I)\cong\k$ and therefore we have $p-|I|-1=d-2$ in Theorem~\ref{zkcoh} (*) and so $p+q-|I|-|J|-1\geq{d-1}$ in (*), which is impossible.\\

{\noindent{{\bf B.}}} Suppose $v$ is either in $V_{1}\cup{V_2}$ or a new vertex, which is not a vertex of exactly one added pyramid over a facet of $P^*$.\\ 

Then we can use minimal non-Golodness for truncation polytopes (see~\cite{B-J}) and induction by $k$ for the case $r=2$: we see, that $K'$ decomposes as ${(K_{1}-\,v_1)}\cup_{\sigma_1}\ldots\cup_{\sigma_l}{(K_{l}-\,v_l)}$, gluing along a common simplex $\sigma_i$, where $K_{i}$ is either a boundary complex of $\vc^{p}(\Delta^{n_1}\times\Delta^{n_2})$ with $p<k$ or a boundary complex of $\vc^{p}(\Delta^{n})$. Therefore, using Proposition~\ref{glue}, we obtain that $K'$ is Golod.\\  

As $K_P$ is not Golod itself and we proved that $K'=K_P-\,v$ is Golod for any $v$, it follows that $K_P$ is minimally non-Golod in the case $r=2$.\\

The case $r=1$ can be proved with the same argument as $r=2$. The only difference is that we do not need to consider generalized truncation polytopes $\vc^{p}(\Delta^{n_1}\times\Delta^{n_2})$ in the case {\bf{B}}.\\

It remains to consider the case $r\ge 3$.\\

For $k=0$ it is clear that $\mathcal Z_{K_{P}-\,v}$ will be homotopy equivalent to a product of $\ge 2$ spheres for any $v$ and therefore $K=K_P$ is not minimally non-Golod itself anymore.\\

For $k\ge 1$ consider $v\in\Delta^{n_1}$. Then the nontrivial multiplication for cohomology classes of $\partial(\Delta^{n_2})$ and $\partial(\Delta^{n_3})$ in the cohomology ring of $\mathcal Z_K$ (by Theorem~\ref{zkcoh}) remains nontrivial in $K'=K-\,v$ and then $K'$ is not Golod.
\end{proof} 

\begin{rema}
The case $r=1$ of Theorem~\ref{maingolod} was proved in~\cite[Theorem 6.19]{B-J}; the crucial point in~\cite{B-J} was to prove combinatorially the 'strong gcd-condition' for $K_P-\,v$ complexes, where $P$ is a truncation polytope. In that proof the fact that the complex dual to a truncation polytope (a stacked polytope) is a union of 'stacks' was used. For generalized truncation polytopes this no longer takes place.
\end{rema}

\begin{exam}
For a vertex cut of a 3-cube $P=\vc^1(\Delta^1\times\Delta^1\times\Delta^1)$, suppose $v_7$ is the new vertex in $K_P$. Then the complex $K'=K_P-\,v_7$ is {\bf{not}} Golod, as there is an induced 4-cycle in its 1-skeleton $sk^1(K')$. The latter is not a chordal graph, which is necessary for $K'$ to be Golod (\cite[Prop. 6.4]{B-J}). 
\end{exam}
According to~\cite[Theorem 2.2]{G-LdM}, the moment-angle manifold $\zp$ corresponding to a generalized truncation polytope $P$ with $r=2,n_{1}\ge{n_{2}}>1,k\ge 0$ is diffeomorphic to a connected sum of sphere products with 2 spheres in each product, if $m=n_{1}+n_{2}+2+k<3(n_{1}+n_{2})=3d$, that is, for all $0\leq k<2(n_{1}+n_{2}-1)$. It is easy to see that Betti numbers of $\zp$ obtained from this description are equal to those calculated using Theorem~\ref{maintr3} and formula $(1.3)$. 

\begin{exam}\label{trconnsum}
{\noindent{\bf{1.}}} Consider $P=\vc^1(\Delta^4\times\Delta^3)$ with $d=7,\,m=10$. Then we get: 
$$\zp\cong 2S^3\times S^{14}\#S^4\times S^{13}\#S^7\times S^{10}\#S^8\times S^9.$$
Here $b_{5}(\zp)=b_{6}(\zp)=0$ and $\zp$ is {\bf{not}} homotopy equivalent to any $\zp$ for truncation polytopes, see Theorem~\ref{b-m} below. \\ 
{\noindent{\bf{2.}}} Consider $P=\vc^1(\Delta^1\times\Delta^1\times\Delta^1)$.
Using Theorem~\ref{zkcoh} it is proved in~\cite[Example 11.5]{B-M} that the corresponding moment-angle manifold $\zp$ is {\bf{not}} homotopy equivalent to a connected sum of products of any number of spheres. See also Example on pp. 26--27 of~\cite{G-LdM}.   
\end{exam}

The topological types of moment-angle manifolds corresponding to truncation polytopes can be described by means of the following result firstly obtained by D.McGavran in~\cite{McG}:

\begin{theo}[{see \cite[Theorem 6.3]{B-M}}]\label{b-m}
Let $P=\vc^{k}(\Delta^{n})$ for $n\ge 2$ be a truncation polytope. Then the
corresponding moment-angle manifold $\zp$ is
diffeomorphic to the connected sum of sphere products:
\[
  \mathop{\#}_{j=1}^{k}
  \bigl(S^{j+2}\times S^{2n+k-j-1}\bigr)^{\#j\binom{k+1}{j+1}},
\]
where $X^{\#k}$ denotes the connected sum of $k$ copies of~$X$.
\end{theo}

\begin{rema}

{\noindent{\bf 1.}}
The author doesn't know any example of a simple polytope $P$ s.t. $\zp$ is topologically a {\bf{nontrivial}} connected sum of sphere products including a product of {\bf{three}} spheres (see also~\cite{G-LdM}).

{\noindent{\bf 2.}}
It is proved in~\cite[Theorem 1.3]{G-LdM} that a moment-angle manifold corresponding to an {\bf{even}} dimensional dual neighbourly polytope is diffeomorhic to a connected sum of sphere products. Below we prove minimal non-Golodness for the Stanley--Reisner rings in the case of even dimensional dual neighbourly polytopes. 
\end{rema}

\begin{prop}\label{neighbour}
If $P$ is an $n$-dimensional dual neighbourly polytope, $n$ is even, then $K_P$ is minimally non-Golod.
\end{prop}
\begin{proof}
By definition, every $\lbrack{\frac{n}{2}}\rbrack$ vertices of $K_P$ form a simplex. Therefore, by Theorem~\ref{zkcoh} any nontrivial full subcomplex cohomology group of positive dimension in the formula (*) is of dimension $\geq\lbrack{\frac{n}{2}}\rbrack-1$. Then for the product in cohomology by (*) we get a nontrivial cohomology class only in dimension $\geq 2(\lbrack{\frac{n}{2}}\rbrack-1)+1=(n-1)$ for even $n$. But this leads to a contradiction, as $(n-1)$-dimensional cohomology group of any full subcomplex of $K'$ is obviously trivial.
\end{proof}

\begin{exam}
By Theorem~\ref{maintr} (b) we see that for generalized truncation polytopes $P$ with $k\geq 1$: $b_3(\zp)=\beta^{-1,4}(P)\neq 0$. On the other hand, if $P$ is a dual neighbourly $n$-polytope, $n\geq 4$, then by~\cite[Proposition 7.34.2]{B-P} and the Hurewicz Theorem one has $b_3(\zp)=\ldots=b_{2[\frac{n}{2}]}(\zp)=0$. Therefore $\zp$ for generalized truncation polytopes $P$ with $k\geq 1$ are {\bf{not}} homotopy equivalent to any $\zp$ for dual neighbourly polytopes $P$.  
\end{exam}

The following statement follows easily from Theorem~\ref{zkcoh}, Theorem~\ref{maintr3}, Theorem~\ref{b-m} and Remark 2 above:

\begin{coro}
In the class of all polytopes $P$ which are either generalized truncation polytopes with $r=1$ and $r=2,0\leq k<2(n_{1}+n_{2}-1)$ or even dimensional dual neighbourly polytopes, the set of all bigraded Betti numbers of $P$ determines the cohomology ring of $\zp$ up to ring isomorphism and the topological type of $\zp$ up to diffeomorphism.
\end{coro}

This is somehow opposite to an interesting example due to S.Choi, see~\cite{Ch}.

The cases $r=1,k\geq 1$ and $r=2$ in Theorem~\ref{maingolod} can be generalized by means of the following result:

\begin{theo}\label{truncgolod}
If $K_P$ is minimally non-Golod, then the same is true for its vertex truncation $Q=\vc(P)$.
\end{theo}
\begin{proof}
Denote $Q=\vc(P)$, $KK=K_P$, $K=K_Q$, $K'=K-\,v'$, $v$ is a new vertex over a facet $F$ of $P^*$ and $d$ is the dimension of $P$.
Consider the following 3 cases:\\

{\noindent{\textit{Case 1:}}} $v'=v$\\

From Theorem~\ref{zkcoh} and minimally non-Golodness of $KK=K_P$ we get that $K'$ is Golod as $(d-1)$-dimensional cohomology group of $\mathcal Z_{K'}$ is obviously trivial (we should only check a pair $I,J$ s.t. $I\cup J=V(KK)$ and $I\cap J=\varnothing$; so either $V(F)$ is in one of the vertex subsets $I$ or $J$, say, $V(F)\subset I$, or $I,J\cap V(F)\neq\varnothing$. Here we use the ring structure described in Theorem~\ref{zkcoh}).\\

{\noindent{\textit{Case 2:}}} $v'\in V(F)$\\

In this case one can see easily that $K'=\Delta^{d-1}\cup_{\Delta^{d-2}}(KK-\,v')$, that is gluing in a simplex of two Golod complexes. Therefore, by Proposition~\ref{glue} we get that $K'$ is Golod.\\

{\noindent{\textit{Case 3:}}} $v'\notin{v\cup V(F)}$\\

Suppose a pair $I,J$ gives a nontrivial product in the cohomology ring of $\mathcal Z_{K'}$. First assume that $v\notin I$ and $v\notin J$. Then $I,J\subset V(KK)$ and either $V(F)\subset I$ or $V(F)\cap I,J\neq\varnothing$. In the latter case $K'_{I}=KK_{I}$ and $K'_J=KK'_J$ and the product in cohomology is trivial as $(d-1)$-dimensional cohomology group of any full subcomplex of $K'$ is trivial. In the former case we get a contradiction by the same argument.\\
Therefore, $v\in I$ or $v\in J$. Without loss of generality, assume $v\in I$. If $V(F)\subset I$ then $K'_I$ and $KK_I$ are topologically equivalent and $K'_J=KK_J$, so the product in cohomology is trivial. If $V(F)\cap I,J\neq\varnothing$ then obviously $K'_I$ is homotopy equivalent to $KK_I$ and $K'_J=KK_J$ and we also get a contradiction. The last possible case is $V(F)\subset J$. We have $K'_{I}=KK_I\sqcup{pt}$ and $K'_J=KK_J-\,F$ and the product in cohomology of $\mathcal Z_{K'}$ can be nontrivial only if we multiply the basis cochains corresponding to the new vertex $v\in I$ and a simplex $M$ in the boundary of $F$ (see Proposition~\ref{glue}). But this is the case of a boundary complex of $\vc^{1}(\Delta^{d})$ without one vertex. The latter complex is a Golod one as we already know, and so the product in $H^{*}(\mathcal Z_{K'})$ is trivial in the last case. 
\end{proof}

\begin{rema}
The results of Theorem~\ref{truncgolod}, Proposition~\ref{neighbour} and other obserbvations from this section are linked by the following interesting conjecture, firstly appeared as~\cite[Question 3.5]{G-P-T-W}:\\

{\noindent\textit{$\mathcal Z_K$ is topologically equivalent to a connected sum of sphere products, 2 spheres in each product, if and only if $K$ is minimally non-Golod and torsion free}.}\\

As a corollary of our results in this section we get another
argument to think of this conjecture as being true.
One can also compare Theorem~\ref{truncgolod} with an open question stated in~\cite{G-LdM} p.24: is it true that if a moment-angle manifold $\zp$ is a connected sum of sphere products then it remains in this class after a vertex  truncation of the polytope $P$.

% (under an additional condition, see~\cite{G-LdM}, 
\end{rema}

%\begin{quote}
%\large{������ ����������}
%\end{quote}


\begin{thebibliography}{99}
%\bibitem{B}
%Victor M. Buchstaber. \emph{Lectures on toric
%topology}. In \emph{Proceedings of Toric Topology Workshop KAIST
%2008}. Trends in Math.~{\bf10}, no.~1. Information Center for
%Mathematical Sciences, KAIST, 2008, pp.~1--64.

\bibitem{B-J} Alexander Berglund and Michael Jollenbeck. \emph{On the Golod property of Stanley--Reisner rings}, J. Algebra \textbf{315}:1 (2007), 249--273.

\bibitem{B-M}
Fr\'ed\'eric Bosio and Laurent Meersseman. \emph{Real quadrics in
$\mathbb C^n$, complex manifolds and convex polytopes.} Acta
Math.~\textbf{197} (2006), no.~1, 53--127.

\bibitem{TT}
Victor M. Buchstaber and Taras E. Panov. \emph{Toric Topology}, A book project (2013); arXiv:1210.2368.

\bibitem{B-P}
Victor M. Buchstaber and Taras E. Panov. \emph{Torus Actions in
Topology and Combinatorics} (in Russian). MCCME, Moscow, 2004, 272
pages.

\bibitem{Ch} Suyoung Choi. \emph{Different moment-angle manifolds arising from two polytopes having the same bigraded Betti numbers}, Preprint (2012); arXiv:1209.0515.

%\bibitem{Ch-K}
%Suyoung Choi and Jang Soo Kim. \emph{A combinatorial proof of a
%formula for Betti numbers of a stacked polytope.} Electron. J.
%Combin.~\textbf{17} (2010), no.~1, Research Paper~9, 8~pp.;
%arXiv:math.CO/0902.2444.

\bibitem{G-LdM} Samuel Gitler and Santiago Lopez de Medrano. \emph{Intersections of quadrics, moment-angle manifolds and connected sums}, Preprint (2009); arXiv:0901.2580.

\bibitem{G-P-T-W} Jelena Grbic, Taras Panov, Stephen Theriault and Jie Wu. \emph{Homotopy types of moment-angle complexes for flag complexes}, Preprint (2012); arXiv:1211.0873.

%\bibitem{G-T} Jelena Grbic, Stephen Theriault. \textit{The homotopy type of the %complement of a coordinate subspace arrangement.} Topology {\bf 46} (2007, no. %4, 357--396)

\bibitem{Li} Ivan Yu. Limonchenko. \emph{Bigraded Betti numbers of some simple polytopes}, Math. Notes \textbf{94}:3 (2013), 373--388. 

\bibitem{Mac}
\emph{Macaulay 2}. A software system devoted to supporting
research in algebraic geometry and commutative algebra. Available
at {\tt http://www.math.uiuc.edu/Macaulay2/}

\bibitem{McG}
D.McGavran. \emph{Adjacent connected sums and torus actions}, Trans. AMS 251 (1979), pp. 235--254.

\bibitem{P}
Taras Panov. \emph{Cohomology of face rings, and torus actions},
in ``Surveys in Contemporary Mathematics''. London Math. Soc.
Lecture Note Series, vol.~\textbf{347}, Cambridge, U.K., 2008, pp.
165--201; arXiv:math.AT/0506526.

%\bibitem{Pa}
%Taras Panov. \emph{Moment--angle manifolds and complexes}.
%In \emph{Proceedings of Toric Topology Workshop KAIST
%2010}. Trends in Math.~{\bf12}, no.~1. Information Center for
%Mathematical Sciences, KAIST, 2010, pp.~43--69.

\bibitem{S}
Richard P. Stanley. \emph{Combinatorics and Commutative Algebra},
second edition. Progr. in Math.~{\bf 41}. Birkh\"auser, Boston,
1996.

%\bibitem{stas63} James D. Stasheff. \emph{Homotopy associativity of
%H-spaces. I.} Transactions Amer. Math. Soc.~\textbf{108} (1963),
%275--292.

\bibitem{T-H}
Naoki Terai and Takayuki Hibi. \emph{Computation of Betti numbers
of monomial ideals associated with stacked polytopes.} Manuscripta
Math., 92(4): 447--453, 1997.

\end{thebibliography}
\end{document}